\documentclass[11pt,a4paper]{amsart}

\usepackage{amssymb,amsfonts,amsmath}

\usepackage[latin1]{inputenc}
\usepackage{color,tikz}

\newtheorem{theorem}{Theorem}
\newtheorem{proposition}[theorem]{Proposition}
\newtheorem{lemma}[theorem]{Lemma}
\newtheorem{corollary}[theorem]{Corollary}

\newtheorem{remark}[theorem]{Remark}

\def\Even{\mathop{\rm Even}\nolimits}
\def\chr{\mathop{\rm chr}\nolimits}
\def\fchr{\mathop{\rm fchr}\nolimits}
\def\fch{\mathop{\rm FCH}\nolimits}%F\mathcal{C}_h

\makeatletter
\@namedef{subjclassname@2010}{%
\textup{2010} Mathematics Subject Classification}
\makeatother

\begin{document}
%\begin{frontmatter}

\title{Free choosability of the cycle}
\author{Yves Aubry, Jean-Christophe Godin and Olivier Togni}
\address{Institut de Math\'ematiques de Toulon, Universit\'e de Toulon, France\\ Institut de Mathématiques de Marseille, France}
\email{yves.aubry@univ-tln.fr,  godinjeanchri@yahoo.fr}
\address{Laboratoire LE2I UMR CNRS 6306, Universit\'e de Bourgogne, France}
%\email{yves.aubry@univ-tln.fr,  godinjeanchri@yahoo.fr}
\email{ olivier.togni@u-bourgogne.fr}



\date{\today}

\begin{abstract}
A graph $G$ is free $(a,b)$-choosable if for any vertex $v$ with $b$ colors assigned and for any list of colors of size $a$ associated with each vertex $u\ne v$, the coloring can be completed by choosing for $u$ a subset of $b$ colors such that adjacent vertices are colored with disjoint color sets. 
In this note, a necessary and sufficient condition for a cycle to be free $(a,b)$-choosable is given. As a corollary, some choosability results are derived for graphs in which cycles are connected by a tree structure.
\end{abstract}

\maketitle

\section{Introduction}
For a graph $G$, we denote its vertex set by $V(G)$ and edge set by $E(G)$.
A color-list $L$ of a graph $G$ is an assignment of lists of integers (colors) to the vertices of $G$. For an integer $a$, a $a$-color-list $L$ of $G$ is a color-list such that $|L(v)|=a$ for any $v\in V(G)$.

In 1996, Voigt considered the following problem:
let $G$ be a graph and $L$ a color-list  and assume that an arbitrary vertex $v\in V(G)$ is precolored by a color $f\in L(v)$. Under which hypothesis is it always possible to complete this precoloring to a proper color-list coloring ? This question leads to the concept of free choosability introduced by Voigt \cite{voi96}. 

Formally, for a graph $G$, integers $a,b$ and a $a$-color-list $L$ of $G$, an {\em $(L,b)$-coloring} of $G$ is a mapping $c$ that associates to each vertex $u$ a subset $c(u)$ of $L(u)$ such that $|c(u)|=b$ and $c(u)\cap c(v)=\emptyset$ for any edge $uv\in E(G)$. The graph $G$ is {\em $(a,b)$-choosable} if for any $a$-color-list $L$ of $G$, there exists an $(L,b)$-coloring . Moreover, $G$ is {\em free $(a,b)$-choosable} if for any $a$-color-list  $L$, any vertex $v$ and any set $c_0\subset L(v)$ of $b$ colors, there exists an $(L,b)$-coloring $c$ such that $c(v)=c_0$.

As shown by Voigt~\cite{voi96}, there are examples of graphs $G$ that are $(a,b)$-choosable but not free $(a,b)$-choosable. Some related recent results concern defective free choosability of planar graphs~\cite{LiMengLiu}.
We investigate, in the next section, the free-choosability of the first interesting case, namely the cycle. We derive a necessary and sufficient condition for a cycle to be free $(a,b)$-choosable (Theorem \ref{theorem52these}). In order to get  a concise statement, we introduce the free-choice ratio of a graph, in the same way that Alon, Tuza and Voigt  \cite{alontuzavoigt97} introduced the choice ratio (which is equal to the so-called fractional chromatic number).

For any real $x$, let $\fch(x)$ be the set of graphs $G$ which are free $(a,b)$-choosable for all $a,b$ such that $\frac{a}{b} \geq x$:

$$\fch(x)=\Bigl\{ G\ | \ \forall \ \frac{a}{b} \geq x, \ G \ \text{is free}\ (a,b)\text{-choosable}\Bigr\}.$$

Moreover, we can define the free-choice ratio 
 $\fchr(G) $ of a graph $G$ by:
 
$$\fchr(G):=\inf \Bigl\{\frac{a}{b} \ | \ G \  \text{is free}\ (a,b)\text{-choosable} \Bigr\}.$$

\begin{remark} 
Erd\H{o}s, Rubin and Taylor have asked \cite{erdo79} the following question:
If $G$ is $(a,b)$-choosable, and $\frac{c}{d}>\frac{a}{b}$, does it imply that $G$ is $(c,d)$-choosable ?
Gutner and Tarsi have shown \cite{GutnerTarsi} that the answer is negative in general.
If we consider the analogue question for the free choosability, then Theorem~\ref{theorem52these} implies that it is true for the cycle.
\end{remark}

The \textsl{path} $P_{n+1}$ of length $n$ is the graph with vertex set $V=\{v_{0},v_{1},\dots,v_{n}\}$ and edge set $E=\bigcup_{i=0}^{n-1} \{v_{i}v_{i+1}\}$. 
The cycle $C_{n}$  of length $n$ is the graph with vertex set  $V=\{v_{0},\dots,v_{n-1}\}$ and edge set $E=\bigcup_{i=0}^{n-1} \{v_{i}v_{i+1 (mod \ n)}\}$.
To simplify the notation, for a color-list $L$ of $P_n$ or $C_n$, we let $L(i)$ denote $L(v_{i})$ and $c(i)$ denote $c(v_{i})$.

The notion of waterfall color-list was introduced in~\cite{AGT10} to obtain choosability results on the weighted path and then used to prove the $(5m,2m)$-choosability of triangle-free induced subgraphs of the triangular lattice. We recall one of the results from \cite{AGT10} that will be used in this note, with the function $\Even$ being defined for any real $x$ by: $\Even(x)$ is the smallest even integer $p$ such that $p\geq x$.

\begin{proposition}[\cite{AGT10}]
\label{theorem48these}
Let $L$ be a color-list  of $P_{n+1}$ such that $|L(0)|=|L(n)|= b$,
and $|L(i)|=a=2b+e$ for all $i \in \{1,\dots,n-1\}$ (with $e>0$). 
\begin{center}
If $n \geq \Even\Bigl(\frac{2b}{e}\Bigr)$ then $P_{n+1}$ is $(L,b)$-colorable.  
\end{center} 
\end{proposition}

For example, let $P_{n+1}$ be the path of length $n$ with a color-list $L$ such that $|L(0)|=|L(n)|= 4$,
and $|L(i)|=9$ for all $i \in \{1,\dots,n-1\}$. Then the previous proposition tells us that we can find an $(L,4)$-coloring of $P_{n+1}$ whenever $n\geq 8$. In other words, if $n\geq 8$, we can choose 4 colors on each vertex such that adjacent vertices receive disjoint colors sets. If  $|L(i)|=11$ for all $i \in \{1,\dots,n-1\}$, then $P_{n+1}$ is $(L,4)$-colorable whenever $n\geq 4$.
On the other side, there are color-lists $L$ for which $P_{n+1}$ is not $(L,b)$-colorable

% \medskip
% The above result is a starting tool used in \cite{AGT10} to attack McDiarmid and Reed's conjecture claiming that every triangle free induced subgraph of the triangular lattice is $(9,4)$-colorable (hence the values $a=9$ and $b=4$ are somehow ``natural'').
% It is also used in the following to determine the free-choice-ratio of the cycle.
% 

%%%%%%%%%%%%%%%%%%%%%%%%%%%%%%%%%%%%%%%%%%%%%%%%%

%%%%%%%%%%%%%%%%%%%%%%%%%%%%%%%%%%%%%%%%%%%%%%%%%

\section{Free choosability of the cycle}\label{cycle}

%%%%%%%%%%%%%%%
% Deuxime rsultat principal :
%%%%%%%%%%%%%%%
We begin with a negative result for the even-length cycle:

\begin{lemma}
 \label{counterexampleC2n}
For any integers $a,b,p$ such that $p\ge 2$, and $\frac{a}{b} < 2+\frac{1}{p}$, the cycle $C_{2p}$ is not free $(a,b)$-choosable. 
\end{lemma}

\begin{proof}
We construct a counterexample for the free-choosability of $C_{2p}$: let $L$ be the $a$-color-list of $C_{2p}$ such that

$$L(i) = \left \{
\begin{array}{ll}
\{1,\dots,a\}, &  \mbox{if} \ i \in \{0,1\}; \\
\{\frac{i-1}{2}a+1,\dots,\frac{i-1}{2}a+a\}, &  \mbox{if} \ i \not=2p-1 \ \mbox{is odd}; \\
\{b+\frac{i-2}{2}a+1,\dots,b+(\frac{i-2}{2}+1)a\},  & \mbox{if} \ i \ \mbox{is even and} \ i \not= 0; \\
\{1,\dots,b,1+(p-1)a,\dots,1+pa-b-1\}, & \mbox{if} \ i=2p-1.
\end{array}
\right.
$$ 

If we choose  $c_0=\{1,\dots,b\} \subset L(0)$, we can check that there does not exist an  $(L,b)$-coloring of $C_{2p}$ such that $c(0)=c_0$, so $C_{2p}$ is not free $(a,b)$-choosable.
See Figure \ref{fc6} for an illustration when $p=3$, $a=9$ and $b=4$.  
\end{proof}

\begin{figure}
\begin{center}
\begin{tikzpicture}[scale=0.8]
\def \n {5}
\def \radius {2cm}
\def \margin {9} % margin in angles, depends on the radius

\node[draw, circle, minimum size=3pt,inner sep=1pt] at (90:2cm)[label=above:{$\{1, 2,\ldots ,9\}$}]{$v_0$};
\node[draw, circle, minimum size=3pt,inner sep=1pt] at (150:2cm)[label=left:{$\{1,\ldots, 4,19,\ldots,23\}$}]{$v_5$};
\node[draw, circle, minimum size=3pt,inner sep=1pt] at (210:2cm)[label=left:{$\{14, \ldots,22\}$}]{$v_4$};
\node[draw, circle, minimum size=3pt,inner sep=1pt] at (270:2cm)[label=below:{$\{10, \ldots, 18\}$}]{$v_3$};
\node[draw, circle, minimum size=3pt,inner sep=1pt] at (330:2cm)[label=right:{$\{5, \ldots, 14\}$}]{$v_2$};
\node[draw, circle, minimum size=3pt,inner sep=1pt] at (30:2cm)[label=right:{$\{1, \ldots, 9\}$}]{$v_1$};

\foreach \s in {0,...,\n}
{
  \draw[-, >=latex] ({360/(\n +1) * (\s -1)+\margin +30}:\radius) 
    arc ({360/(\n +1) * (\s - 1)+\margin +30}:{360/(\n +1) * (\s)-\margin +30}:\radius);
} 
\end{tikzpicture}
\end{center}
 \caption{\label{fc6}The cycle $C_6$, along with a $9$-color-list $L$ for which there is no $(L,4)$-coloring $c$ such that $c(v_0)=\{1,2,3,4\}$.}
\end{figure}
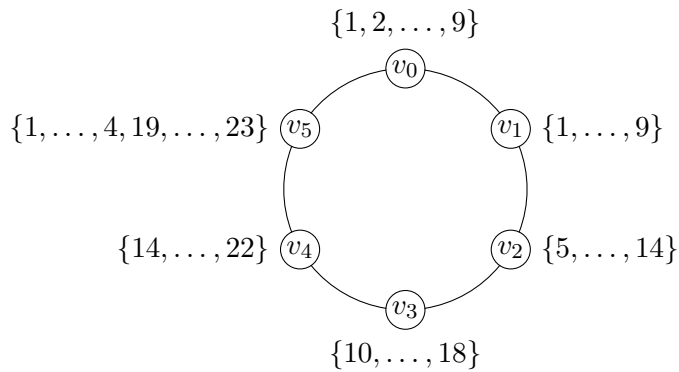

Now, if $\lfloor x \rfloor$ denotes the greatest integer less than or equal to the real $x$,  we can state:

\begin{theorem}
\label{theorem52these}
For the cycle $C_n$ of length $n$,  
$$C_n \in \fch\Big(2+ \Big\lfloor \frac{n}{2} \Big\rfloor ^{-1}\Big). $$
Moreover, we have:
$$\fchr(C_n)=2+ \Big\lfloor \frac{n}{2} \Big\rfloor ^{-1}.$$
\end{theorem}

\begin{proof}
Let $a,b$ be two integers such that  $a/b \geq  2+\lfloor \frac{n}{2}\rfloor ^{-1}$. Let $L$ be a $a$-color-list of $C_n$. Without loss of generality, we can suppose that $v_0$ is the vertex chosen for the free-choosability and $c_0\subset L(v_0)$ has $b$ elements. 
It remains to construct an $(L,b)$-coloring $c$ of $C_n$ such that $c(v_0)=c_0$.
Hence we have to construct an $(L',b)$-coloring  $c$ of $P_{n+1}$  such that $L'(0)=L'(n)=L_0$ and for all $i \in \{1,...,n-1\}$,  $L'(i)=L(v_i)$. We have 
$|L'(0)|=|L'(n)|=b$ and for all  $i \in \{1,...,n-1\}$,  $|L'(i)|=a$.
Since 
$a/b \geq 2+\lfloor \frac{n}{2}\rfloor ^{-1}$ and 
$e=a-2b$, we get 
$e/b \geq\lfloor \frac{n}{2}\rfloor^{-1}$ hence
$n \geq \Even(2b/e)$.
Using Proposition \ref{theorem48these}, we get:

$$C_n \in \fch(2+ \Big\lfloor \frac{n}{2} \Big\rfloor ^{-1}).$$

Hence, we have that $\fchr(C_n) \leq 2+\lfloor \frac{n}{2}\rfloor ^{-1}.$ Moreover, let us prove that 
 $M=2+\lfloor \frac{n}{2}\rfloor ^{-1}$ is reached. For $n$ odd, Voigt has proved \cite{voi98} that the choice ratio $\chr(C_n)$ of a cycle of odd length $n$ is exactly $M$.
% $$\chr(G):=\inf \{\frac{a}{b} \ | \ G \  \text{is}\ (a,b)\text{-choosable} \}=M.$$
Hence $\fchr(C_n)\geq \chr(C_n)=M$, and the result is proved.
For $n$ even, Lemma~\ref{counterexampleC2n} asserts that $C_n$ is not free $(a,b)$-choosable for $\frac{a}{b} < 2+\lfloor \frac{n}{2}\rfloor ^{-1}$.
\end{proof}

\begin{remark} In particular, the previous theorem implies that
if $b,e,n$ are integers such that $n \geq \Even(\frac{2b}{e})$, then the cycle $C_{n}$ of length $n$ is free $(2b+e,b)$-choosable.
\end{remark}

In order to extend the result to other graphs then cycles, the following simple proposition will be useful:
\begin{proposition}\label{propGv}
 Let $a,b$ be integers with $a\ge 2b$. Let $G$ be a graph and $G_v$ be the graph obtained by adding a leave $v$ to any vertex of $G$. Then $G$ is free $(a,b)$-choosable if and only if $G_v$ is free $(a,b)$-choosable.
\end{proposition}
\begin{proof}Since the "only if" part holds trivially, let us prove the "if" part.
Assume $G$ is free $(a,b)$-choosable and let $L$ be a $a$-color-list of $G$. Let $v$ be a new vertex and let $G_v$ be the graph obtained from $G$ by adding the edge $uv$, for some $u\in V(G)$. Then any $(L,b)$-coloring $c$ of $G$ can be extended to an $(L,b)$-coloring of $G_v$ by giving to $v$ $b$ colors from $L(v)\setminus c(u)$ ($|L(v)\setminus c(u)|\ge b$ since $a\ge 2b$). If $v$ is colored with $b$ colors from its list, then, since $G$ is free $(a,b)$-choosable, the coloring can be extended to an $(L,b)$-coloring of $G_v$ by first choosing for $u$ a set of $b$ colors from $L(u)\setminus c(v)$.
\end{proof}

Starting from a single edge and applying inductively Proposition~\ref{propGv} allows to obtain the following corollary:
\begin{corollary}\label{corT}
Let $T$ be a tree of order $n\ge 2$. Then $$T\in \fch (2).$$  
\end{corollary}

Now, we can state the following:
\begin{proposition}
 If $G$ is a unicyclic graph with girth $g$, then $$G \in \fch\Big(2+ \Big\lfloor \frac{g}{2} \Big\rfloor ^{-1}\Big). $$
\end{proposition}
\begin{proof}
Let $a,b$ be two integers such that  $a/b \geq  2+\lfloor \frac{n}{2}\rfloor ^{-1}$, $L$ be a $a$-color-list of $G$, $C=v_1,\ldots,v_g$ be the unique cycle (of length $g$) of $G$ and $T_i$, $i\in\{1,\ldots, g\}$, be the subtree of $G$ rooted at vertex $v_i$ of $C$.

Let $v$ be the vertex chosen for the free choosability and let $c_0\subset L(v)$ be a set of cardinality $b$. If $v\in C$, then by virtue of Theorem~\ref{theorem52these}, there exists an $(L,b)$-coloring $c$ of $C$ such that $c(v)=c_0$. This coloring can be easily extended to the whole graph by coloring the vertices of each tree $T_i$ thanks to Corollary~\ref{corT}.
If $v\in T_i$ for some $i$, $1\le i\le g$, then Corollary~\ref{corT} asserts that the coloring can be extended to $T_i$. Then color $C$ starting at vertex $v_i$ by using Theorem~\ref{theorem52these}.
Finally, complete it by coloring each tree $T_j$, $1\le j\ne i\le g$.
\end{proof}

\section{Applications}
As an example to the possible use of the results from Section~\ref{cycle}, we begin with determining the free choosability of a binocular graph, i.e. two cycles linked by a path.

For integers $m,n$ and $p$ such that $m,n\ge 3$ and $p\ge 0$, the {\em binocular graph} $\text{BG}(m,n,p)$ is the disjoint union of an $m$-cycle $u_0,u_1,\ldots,u_{m-1}$ and of an $n$-cycle $v_0,\ldots ,v_{n-1}$ with vertices $u_0$ and $v_0$ linked by a path of length $p$ given by $u_0,x_1,\ldots, x_{p-1},v_0$. Note that if $p=0$, then $u_0$ and $v_0$ are the same vertex.

\begin{proposition}
 For any $m\ge 3$, $n\ge 3$ and $p\ge 0$, $$\text{BG}(m,n,p) \in \fch \Big(2+ \Big\lfloor \frac{\min(m,n)}{2} \Big\rfloor ^{-1}\Big). $$
\end{proposition}
\begin{proof}
Assume without loss of generality that $m\ge n$ and let $a,b$ be integers such that  $a/b \geq  2+\lfloor \frac{n}{2}\rfloor ^{-1}$. Let $L$ be a $a$-color-list of $\text{BG}(m,n,p)$.
Let $y$ be the vertex chosen for the free choosability and let $c_0\subset L(y)$ be a set of cardinality $b$. If $y$ lies on the $m$-cycle, then by virtue of Theorem~\ref{theorem52these}, there exists an $(L,b)$-coloring $c$ of the $m$-cycle such that $c(y)=c_0$. By Corollary~\ref{corT}, this coloring can be extended to the vertices of the path. Now, it remains to color the vertices of the $n$-cycle, with $v_0$ being already colored. This can be done thanks to Theorem~\ref{theorem52these}.
If $y$ lies on the $n$-cycle, the argument is similar.
If $y\in \{x_1,\ldots, x_{p-1}\}$, then the coloring can be extended to the whole path and the coloring of the $m$-cycle and $n$-cycle can be completed thanks to Theorem~\ref{theorem52these}.
\end{proof}

This method can be extended to prove similar results on graphs with more than two cycles, connected by a tree structure.

Define a {\em tree of cycles} to be a graph $G$ such that all its cycles are disjoint and collapsing all vertices of each cycle of $G$ produces a tree.

%let an {\em $\ell$-cycle expansion} of a graph $G$ be any graph obtained by replacing in $G$ each vertex $x$ of a subset $X\in V(G)$ by a cycle $C^x$ of order at least $\ell$, with vertices adjacent to $x$ in $G$ being adjacent to any vertex of $C^x$ in the $\ell$-cycle expansion.

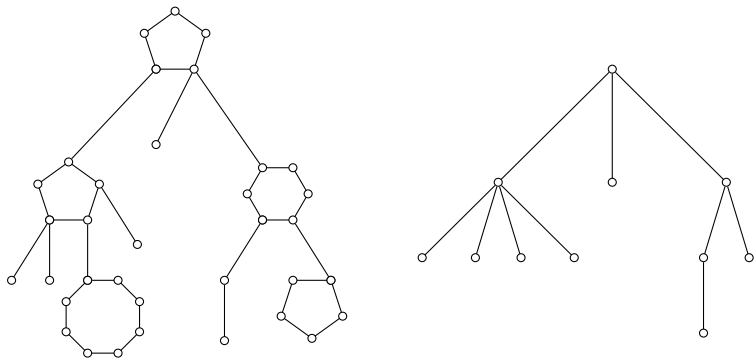
\begin{figure}
\begin{center}
\begin{tikzpicture}[scale=1]
\tikzstyle{every node}=[draw,circle,fill=white,minimum size=3pt,inner sep=1pt]
\draw (8,0) node (r)  {}
     (r) -- ++(1.5cm,-1.5cm) node (x3)  {}
     (r)-- ++(0cm,-1.5cm) node (x2) {}
     (r)-- ++(-1.5cm,-1.5cm) node (x1) {}
     (x1)-- ++(-1cm,-1cm) node (x11) {}
     (x1)-- ++(-.3cm,-1cm) node (x12) {}
     (x1)-- ++(.3cm,-1cm) node (x11) {}
     (x1)-- ++(1cm,-1cm) node (x12) {}
     (x3)--++(-.3cm,-1cm) node (x31) {}
     (x3)--++(.3cm,-1cm) node (x32) {}
     (x31)--++(0cm,-1cm) node () {};

\draw (2,0) node (r)  {}
     -- ++(0:.5cm) node (r1)  {}
     -- ++(72:.5cm) node (r2) {}
     -- ++(144:.5cm) node (r3) {}
     -- ++(216:.5cm) node (r4) {}
     -- ++(288:.5cm) node (r5) {}
    (r1)--++(-.5cm,-1cm) node (){};
\draw (0.6,-2) node (f)  {}
     -- ++(0:.5cm) node (f1)  {}
     -- ++(72:.5cm) node (f2) {}
     -- ++(144:.5cm) node (f3) {}
     -- ++(216:.5cm) node (f4) {}
     -- ++(288:.5cm) node (f5) {}
     (r)--++ (f3)
     (f)--++(-.5cm,-.8cm) node (){}
     (f)--++(0cm,-.8cm) node (){}
     (f1)--++(0cm,-.8cm) node (f12){}
     (f2)--++(0.5cm,-.8cm) node (){}
     (f12)--++(0:.4) node () {}
          --++(-45:.4) node () {}
          --++(-90:.4) node () {}
          --++(-135:.4) node () {}
          --++(180:.4) node () {}
          --++(135:.4) node () {}
          --++(90:.4) node () {}
          --++(45:.4) node () {}
          ;

\draw (3.4,-2) node (g)  {}
     -- ++(0:.4cm) node (g1)  {}
     -- ++(60:.4cm) node (g2) {}
     -- ++(120:.4cm) node (g3) {}
     -- ++(180:.4cm) node (g4) {}
     -- ++(240:.4cm) node (g5) {}
     -- ++(300:.4cm) node (g6) {}
     (r1)--++(g4)
     (g)--++(-0.5cm,-.8cm) node (fg) {}
     (fg)--++(0cm,-.8cm) node () {}
     (g1)--++(0.5cm,-.8cm) node (g11) {}
     -- ++(-72:.5cm) node () {}
     -- ++(-144:.5cm) node () {}
     -- ++(144:.5cm) node () {}
     -- ++(72:.5cm) node () {}
     -- ++(0:.5cm) node () {};
\end{tikzpicture}
\end{center}
 \caption{\label{flc}A tree of cycles (on the left), and the associated tree obtained by collapsing cycles (on the right).}
\end{figure}

\begin{corollary}
Any tree of cycles of girth $g$ is in $\fch (2+ \lfloor \frac{g}{2} \rfloor ^{-1})$.
\end{corollary}

\section{Algorithmic considerations}
Let $n\ge 3$ be an integer and let $a,b$ be two integers such that  $a/b \geq  2+\lfloor \frac{n}{2}\rfloor ^{-1}$. Let $L$ be a $a$-color-list of $C_n$.

As defined in~\cite{AGT10}, a {\em waterfall} list $L$ of a path $P_{n+1}$ of length $n$ is a list $L$ such that for all $i,j \in \{0,\dots,n\}$ with $|i-j| \geq 2$, we have $L(i) \cap L(j) = \emptyset$.
Le  $m=|\cup_{i=0}^n L(i)|$ be the total number of colors of the color-list $L$.

The algorithm behind the proof of Proposition~\ref{theorem48these} consists in three steps: first, the transformation of the list $L$ into a waterfall list $L'$ by renaming some colors; second, the construction of the $(L',b)$-coloring by coloring vertices from 0 to $n-1$, giving to vertex $i$ the first $b$-colors that are not used by the previous vertex; third, the backward transformation to obtain an $(L,b)$-coloring from the $(L',b)$-coloring by coming back to original colors and resolving color conflicts if any.
It can be seen that the time complexity of the first step is $O(mn)$; that of the second one is $O(a²n)$ and that of the third one is $O(\max(a,b^3)n)$. Therefore, the total running time for computing a free $(L,b)$-coloring of the cycle $C_n$ is $O(\max(m,a^2,b^3)n)$.
%The number of steps required to find an $(L,b)$-coloring if it exists then depends on the size $m'$ of the waterfall list $L^c$.
%In fact, since any color $k$ of $A(L)$ can be renamed at most $\lceil n/2\rceil$ times (if it is present on the list of all vertices), the size $m'$ of $L^c$ can be upper bounded by $m\lceil n/2\rceil$. 
%We have also to save, at each step, the color renumbering, in order to transform at the end the $(L^c,b)$-coloring into an $(L,b)$-coloring. This can be done simply with an array of size $m'$.

%In order to improve the time needed to compute an $(L,b)$-coloring of $C_n$, we can actually relax some constraints on the waterfall list by requiring only that for all $i\in \{0,\dots,n-1\}$, $L(i) \cap L(i+1)\cap L(i+2) = \emptyset$. This allows to reduce the number of steps of the renumbering process while transforming $L$ into $L^c$.

%%%%%%%%%%%%%%%%%%%%%%%%%%%%%%%%%%%%%%%%%%%%%%%%%%%%%%%%%%%%%%%%%%%%%%%%%%%%%%%%%%%%%%%%%%%%%%%%%%%%%%%%%%%%%%%
%%%%%%%%%%%%%%%%%%%%%%%%%%%%%     BIBLIOGRAPHIE              %%%%%%%%%%%%%%%%%%%%%%%%%%%%%%%%%%%%%%%%%%%%%%%%%%
%%%%%%%%%%%%%%%%%%%%%%%%%%%%%%%%%%%%%%%%%%%%%%%%%%%%%%%%%%%%%%%%%%%%%%%%%%%%%%%%%%%%%%%%%%%%%%%%%%%%%%%%%%%%%%%

\end{document}